\documentclass[11pt]{amsart}
\usepackage{graphicx}
\usepackage[mathcal]{euscript}
\usepackage{float}
\usepackage{cite}
\usepackage{verbatim}
\usepackage{url}

\restylefloat{figure}

\newtheorem{theorem}{Theorem}[section]
\newtheorem{lemma}[theorem]{Lemma}

\newtheorem{proposition}[theorem]{Proposition}

\theoremstyle{definition}

\newtheorem{example}[theorem]{Example}

\theoremstyle{remark}

\numberwithin{equation}{section}

\def\alt{\operatorname{alt}}
\def\dalt{\operatorname{dalt}}

\def\adj{\operatorname{adj}}

\usepackage{tikz}

\usetikzlibrary{arrows,shapes,decorations,backgrounds,patterns}

\pgfdeclarelayer{background}
\pgfdeclarelayer{background2}
\pgfdeclarelayer{background2a}
\pgfdeclarelayer{background2b}
\pgfdeclarelayer{background3}
\pgfdeclarelayer{background4}
\pgfdeclarelayer{background5}
\pgfdeclarelayer{background6}
\pgfdeclarelayer{background7}

\pgfsetlayers{background7,background6,background5,background4,background3,background2b,background2a,background2,background,main}

\definecolor{lsupurple}{RGB}{70,29,124}
\definecolor{lsugold}{RGB}{253,208, 35}

\begin{document}

\title{Invariants for Turaev genus one links}

\author{Oliver T. Dasbach}
\address{Department of Mathematics\\
Louisiana State University\\
Baton Rouge, LA}
\email{kasten@math.lsu.edu}

\author{Adam M. Lowrance}
\address{Department of Mathematics\\
Vassar College\\
Poughkeepsie, NY} 
\email{adlowrance@vassar.edu}
\thanks{The first author is supported in part by NSF grant DMS-1317942\\
\indent The second author is supported by Simons Collaboration Grant for Mathematicians no. 355087.}

\date{}

\begin{abstract}
The Turaev genus defines a natural filtration on knots where Turaev genus zero knots are precisely the alternating knots.
We show that the signature of a Turaev genus one knot is determined by the number of components in its all-$A$ Kauffman state, the number of positive crossings, and its determinant. We also show that either the leading or trailing coefficient of the Jones polynomial of a Turaev genus one link (or an almost alternating link) has absolute value one. 

\end{abstract}

\maketitle

\section{Introduction}

Tait's flyping theorem, proven by Menasco and Thistlethwaite \cite{MenTh:Flype}, gives a classification of alternating links in terms of their alternating projections. Alternating links have a natural generalization by allowing alternating projections on surfaces other than the sphere.
For each link diagram, Turaev \cite{Turaev:SimpleProof, DFKLS:Jones} constructed a closed, orientable surface on which the link projects alternatingly.
The smallest genus among all Turaev surfaces of a given link is the {\it Turaev genus}, and links of Turaev genus zero are precisely the alternating links. The aim of this paper is to study two invariants for links of Turaev genus one: the signature and the Jones polynomial.

The signature $\sigma(K)$ of a knot $K$ was originally defined by Trotter \cite{Trotter:Signature}. Milnor \cite{Milnor:InfiniteCyclic} found an alternate definition of the signature of a knot using the infinite cyclic cover of the knot complement, and Erle \cite{Erle:Signature} proved that Trotter and Milnor's constructions are equivalent. Murasugi \cite{Murasugi:Signature} extended the definition of signature to links of more than one component and showed that signature gives lower bounds on the slice genus and unknotting number of a knot. Kauffman and Taylor \cite{KauffmanTaylor:Signature} showed that the signature of a link is a concordance invariant.

The signature of a link $L$ can be defined as the signature of a quadratic form associated to a Seifert surface of $L$, i.e. an oriented surface whose boundary is $L$. Gordon and Litherland \cite{GordonLitherland:Signature} showed how to compute the signature of a knot from a quadratic form associated to the (possibly non-orientable) checkerboard surfaces of a diagram of $L$. Traczyk \cite{Traczyk:Signature} used the Gordon-Litherland formulation of the signature to compute the signature of non-split alternating links. Suppose that $L$ is a non-split alternating link with alternating diagram $D$. Let $s_A(D)$ and $s_B(D)$ denote the number of components in the all-$A$ and all-$B$ state of $D$, as in Figure \ref{figure:resolution}. Also, let $c_+(D)$ and $c_-(D)$ denote the number of positive and negative crossings in $D$, as in Figure \ref{figure:crossingsign}. Then
\begin{equation}
\label{equation:signaturealternating}
\sigma(L) = s_A(D) - c_+(D) -1 = -s_B(D) + c_-(D) + 1.
\end{equation}

Lee \cite{Lee:KhovanovAlternating} used Equation \ref{equation:signaturealternating} to prove that the reduced Khovanov homology of a non-split alternating link $L$ is supported entirely in the $\delta$-grading of $-\sigma(L)/2$.  Rasmussen \cite{Rasmussen:Slice} defined a concordance invariant $s$ from Khovanov homology and used Lee's result to show that if $K$ is an alternating knot, then $s(K)=-\sigma(K)$. In a similar vein, Ozsv\'ath and Szab\'o \cite{OS:Alternating} showed that the knot Floer homology $\widehat{HFK}(K)$ of an alternating knot $K$ is supported in the $\delta$-grading of $-\sigma(K)/2$. The $\tau$-invariant is a concordance invariant arising from the Heegaard Floer package, and Ozsv\'ath and Szab\'o \cite{OS:Tau} showed that if $K$ is alternating, then $\tau(K) = -\sigma(K)/2$.

In \cite{DasLow:TuraevConcordance} the authors investigated the relationship between the signature of a knot and the maximum and minimum $\delta$-gradings in Khovanov and knot Floer homology. We showed that for any diagram $D$ of a knot $K$, the following inequality holds:
\begin{equation}
\label{inequality:TuraevSignature}
s_A(D) - c_+(D) -1 \leq \sigma(K) \leq -s_B(D) + c_-(D) + 1.
\end{equation}

\begin{figure}[h]
$$\begin{tikzpicture}[scale=.8, >=triangle 45]
\draw[->] (-1,-1) -- (1,1);
\draw[->] (-.25,.25) -- (-1,1);
\draw (.25,-.25) -- (1,-1);
\draw (0,-1.5) node{\Large{$+$}};

\begin{scope}[xshift = 5cm]
\draw [->] (1,-1) -- (-1,1);
\draw [->] (.25,.25) -- (1,1);
\draw (-.25,-.25) --(-1,-1);
\draw (0,-1.5) node{\Large{$-$}};
\end{scope}
\end{tikzpicture}$$
\caption{Positive and negative crossings in a link diagram.}
\label{figure:crossingsign}
\end{figure}
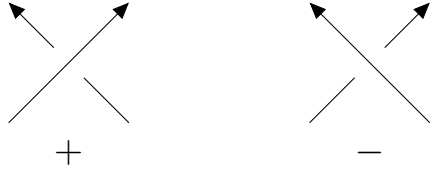

Define the determinant $\det L$ of the link $L$ by $\det L =|\Delta_L(-1)|$ where $\Delta_L(t)$ is the Alexander polynomial of $L$. 
The Turaev surface $F(D)$ of a link diagram $D$ is a closed, oriented surface whose construction is given in Section \ref{section:Turaev}. The genus of the Turaev surface of $D$ is zero if and only if $D$ is the connected sum of alternating diagrams (in which case the associated link is alternating).  
The first main theorem of this article gives a formula for the signature of a knot with a diagram whose Turaev surface has genus one.
\begin{theorem}
\label{theorem:mainsignature}
Let $K$ be a knot with diagram $D$ whose Turaev surface has genus one. The signature of $K$ is determined by
$$\sigma(K) = s_A(D) - c_+(D) \pm 1~\text{and}~\sigma(K) \equiv \det(K) -1 \mod 4.$$
\end{theorem}
The two conditions in Theorem \ref{theorem:mainsignature} determine the signature of $K$ because the determinant of a knot is always odd and its signature is always even. In Section \ref{section:signature}, we give a formulation of Theorem \ref{theorem:mainsignature} for links. 

An $n$-tangle $R$ is an embedding of $n$ arcs and $m$ circles into a $3$-ball for $n>0$ and $m\geq 0$. An $n$-tangle diagram is a regular projection of $R$ inside of a round circle with only transverse double points, and an $n$-tangle is called {\em alternating} if it has an alternating diagram. The intersections of the $n$-strands of $R$ with the boundary circle are decorated with $+$ and $-$ signs according to whether the first crossing in $R$ involving that strand is an over-crossing or an under-crossing. A face of a tangle diagram is a connected component of the projection disk minus the boundary circle union the tangle projection. A tangle diagram is called {\em proper} if no face is incident to two or more different arcs in the boundary circle. If a tangle diagram is proper and alternating, then the $+$ and the $-$ decorations must alternate around the boundary circle.
  
Armond and Lowrance \cite{ArmLow:Turaev} and independently Kim \cite{Kim:TuraevClassification} classified link diagrams whose Turaev surface is genus one. Every non-split link of Turaev genus one has a diagram obtained by arranging an even number of proper alternating $2$-tangles into a circle as in Figure \ref{figure:genus1}. Examples of Turaev genus one links include pretzel links and Montesinos links. See Subsection \ref{subsection:altdecomp} for a detailed treatment of this result. 
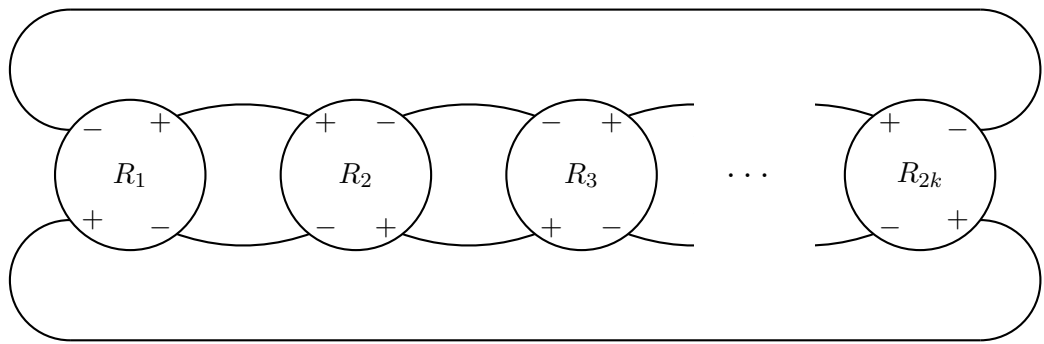
\begin{figure}[h]
$$\begin{tikzpicture}[thick]
\draw [bend left] (0,.5) edge (3,.5);
\draw [bend right] (0,-.5) edge (3, -.5);
\draw [bend left] (3,.5) edge (6,.5);
\draw [bend right] (3,-.5) edge (6, -.5);
\draw [bend left] (6,.5) edge (9,.5);
\draw [bend right] (6,-.5) edge (9, -.5);
\draw [bend left] (7.5,.5) edge (10.5,.5);
\draw [bend right] (7.5,-.5) edge (10.5,-.5);

\fill[white] (7.5,1) rectangle (9.1,-1);
\draw (8.25,0) node{\Large{$\dots$}};

\draw (-.8,.6) arc (270:90:.8cm);
\draw (-.8,-.6) arc (90:270:.8cm);
\draw (11.3,.6) arc (-90:90:.8cm);
\draw (11.3,-.6) arc (90:-90:.8cm);
\draw (-.8,2.2) -- (11.3,2.2);
\draw (-.8,-2.2) -- (11.3,-2.2);

\fill[white] (0,0) circle (1cm);
\draw (0,0) node {$R_1$};
\draw (0,0) circle (1cm);
\fill[white] (3,0) circle (1cm);
\draw (3,0) node {$R_2$};
\draw (3,0) circle (1cm);
\fill[white] (6,0) circle (1cm);
\draw (6,0) node {$R_3$};
\draw (6,0) circle (1cm);
\fill[white] (10.5,0) circle (1cm);
\draw (10.5,0) node {$R_{2k}$};
\draw (10.5,0) circle (1cm);

\draw (-.5,.6) node{$-$};
\draw (-.5,-.6) node{$+$};
\draw (.4,.7) node{$+$};
\draw (.4,-.7) node{$-$};

\draw (2.6,.7) node{$+$};
\draw (2.6,-.7) node{$-$};
\draw (3.4, .7) node{$-$};
\draw (3.4,-.7) node{$+$};

\draw (5.6, .7) node{$-$};
\draw (5.6,-.7) node{$+$};
\draw (6.4,.7) node{$+$};
\draw (6.4,-.7) node{$-$};

\draw (10.1, .7) node{$+$};
\draw (10.1,-.7) node{$-$};
\draw (11,.6) node{$-$};
\draw (11,-.6) node{$+$};

\end{tikzpicture}$$
\caption{Every non-split link of Turaev genus one has a diagram as above. Each $R_i$ is a proper alternating tangle.}
\label{figure:genus1}
\end{figure}

The endpoints of a $2$-tangle $R$ can be connected in two different ways to form a link. If the two northern endpoints are joined and the two southern endpoints are joined, then the resulting link $N(R)$ is called the {\em numerator} of $R$. If the two eastern endpoints are joined and the two western endpoints are joined, then the resulting link $D(R)$ is called the {\em denominator} of $R$. See Figure \ref{figure:tangleclosures}. 
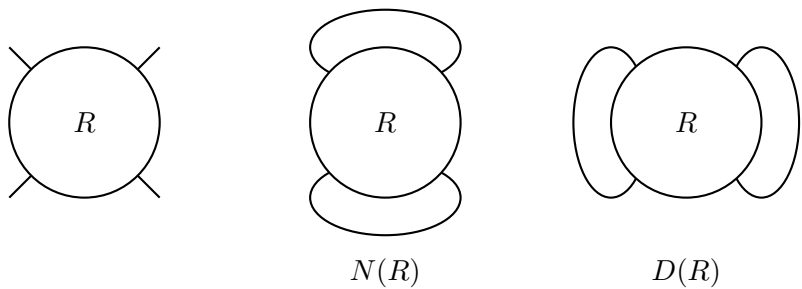
\begin{figure}[h]
$$\begin{tikzpicture}[thick]

\draw (0,0) -- (2,2);
\draw (0,2) -- (2,0);
\fill[white] (1,1) circle (1cm);
\draw (1,1) circle (1cm);
\draw (1,1) node{$R$};

\begin{scope}[xshift=4cm]

\draw (1,2) ellipse (1 cm and .5cm);
\draw (1,0) ellipse (1 cm and .5cm);

\fill[white] (1,1) circle (1cm);
\draw (1,1) circle (1cm);
\draw (1,1) node{$R$};

\draw (1,-1) node {$N(R)$};

\end{scope}

\begin{scope}[xshift=8cm]

\draw (0,1) ellipse (.5 cm and 1 cm);
\draw (2,1) ellipse (.5 cm and 1 cm);

\fill[white] (1,1) circle (1cm);
\draw (1,1) circle (1cm);
\draw (1,1) node{$R$};

\draw (1,-1) node {$D(R)$};

\end{scope}

\end{tikzpicture}$$
\caption{The tangle $R$, its numerator closure $N(R)$, and its denominator closure $D(R)$.}
\label{figure:tangleclosures}
\end{figure}

An orientation of a Turaev  genus one link $L$ yields a direction on each of the edges in the diagram $D$ of Figure \ref{figure:genus1}. The orientation of the strands of $R_i$ inside $D$ is the same as the orientation of the strands of $R_i$ inside either $N(R_i)$ or $D(R_i)$ (or both). Since each $2$-tangle $R_i$ has two incoming edges and two outgoing edges, it follows that the orientation of $R_i$ agrees with the orientation of $N(R_i)$ for each $i=1,\dots, 2k$, or the orientation of $R_i$ agrees with the orientation of $D(R_i)$ for each $i=1,\dots, 2k$. In the first case, we say $D$ has the {\em numerator orientation}, and in the second case, we say $D$ has the {\em denominator orientation}.
\begin{theorem}
\label{theorem:signature2}
Let $L$ be a link with Turaev genus one diagram $D$ as in Figure \ref{figure:genus1}. If $D$ has the numerator orientation, then 
$$\sigma(L) = \pm 1 + \sum_{i=1}^{2k} \sigma(N(R_i)).$$
If $D$ has the denominator orientation, then
$$\sigma(L) = \pm 1 + \sum_{i=1}^{2k} \sigma(D(R_i)).$$
\end{theorem}
As in Theorem \ref{theorem:mainsignature}, if $K$ is a knot, then its signature is determined by Theorem \ref{theorem:signature2} and the fact that $\sigma(K) \equiv \det (K) -1 \mod 4$.

The Jones polynomial \cite{Jones:Polynomial} has been wildly successful at answering difficult questions about diagrammatic properties of knots and links. The first major success of this kind was the proof by Kauffman \cite{Kauffman:Bracket}, Murasugi \cite{Murasugi:JonesConjectures}, and Thistlethwaite \cite{Thistlethwaite:JonesBreadth} that an alternating diagram of a link with no nugatory crossings has the fewest possible number of crossings. Kauffman \cite{Kauffman:Bracket} also proved that if a link is alternating, then the first and last coefficients of the Jones polynomial have absolute value one. Lickorish and Thistlethwaite \cite{LickorishThistlethwaite:Adequate} extended Kauffman's result to the class of adequate links. In our last main result of the paper, we prove a similar result about the Jones polynomial of almost alternating links and links of Turaev genus one. Adams et. al. \cite{Adams:AlmostAlternating} define a link $L$ to be {\em almost alternating} if $L$ is non-alternating, but has a diagram $D$ such that one crossing change transforms $D$ into an alternating diagram. All almost alternating links are Turaev genus one, but it remains an open question whether there exists a Turaev genus one link that is not almost alternating.
\begin{theorem}
\label{theorem:AlmostAlternating}
Let $L$ be an almost alternating link or a link of Turaev genus one with Jones polynomial
$$V_L(t) = a_m t^m + a_{m+1} t^{m+1} +\cdots a_{M-1} t^{M-1} + a_M t^M,$$
where $a_m$ and $a_M$ are nonzero. Either $|a_m| = 1$ or $|a_M|=1$ (or both).
\end{theorem}
Theorem \ref{theorem:AlmostAlternating} provides a computable obstruction to a link being almost alternating or having Turaev genus one. Among the knots with twelve or fewer crossings listed in KnotInfo \cite{KnotInfo}, there are 35 unknown values for Turaev genus. This obstruction shows that 12 of these knots cannot be almost alternating or Turaev genus one. 

This paper is organized as follows. Section \ref{section:Turaev} contains the construction of the Turaev surface. In Section \ref{section:signature}, we prove Theorems \ref{theorem:mainsignature} and \ref{theorem:signature2}. Finally in Section \ref{section:Jones}, we prove Theorem \ref{theorem:AlmostAlternating} and use it to prove that a collection of knots have Turaev genus at least two.

\section{The Turaev surface}
\label{section:Turaev}

In this section, we discuss the Turaev surface of a link diagram, the Turaev genus of a link, and connections between the Turaev surface, Turaev genus, and other knot and link invariants. Champanerkar and Kofman \cite{CK:Survey} provide an excellent recent survey article on this topic.

\subsection{Construction of the Turaev surface}

Each crossing of a link diagram $D$ has an $A$-resolution and a $B$-resolution, as depicted in Figure \ref{figure:resolution}. A {\em state} of $D$ is the set of curves obtained by performing either an $A$-resolution or a $B$-resolution for each crossing. The all-$A$ state (or all-$B$ state) is the state obtained by performing an $A$-resolution (or a $B$-resolution) for every crossing in $D$. Let $s_A(D)$ and $s_B(D)$ denote the number of components in the all-$A$ and all-$B$ states of $D$ respectively. The {\em trace} of each resolution is a small line segment connecting the two arcs of the resolution. 
\begin{figure}[h]
$$\begin{tikzpicture}[>=stealth, scale=.8]
\draw (-1,-1) -- (1,1);
\draw (-1,1) -- (-.25,.25);
\draw (.25,-.25) -- (1,-1);
\draw (-3,0) node[above]{\Large{$A$}};
\draw[->,very thick] (-2,0) -- (-4,0);
\draw (3,0) node[above]{\Large{$B$}};
\draw[->,very thick] (2,0) -- (4,0);
\draw (-5,1) arc (120:240:1.1547cm);
\draw (-7,-1) arc (-60:60:1.1547cm);
\draw (5,1) arc (210:330:1.1547cm);
\draw (7,-1) arc (30:150:1.1547cm);
\draw[blue,very thick] (-5.57735,0) -- (-6.4226,0);
\draw[red,very thick] (6,0.422625) -- (6,-0.422625);
\end{tikzpicture}$$
\caption{The resolutions of a crossing and their traces in a link diagram.}
\label{figure:resolution}
\end{figure}
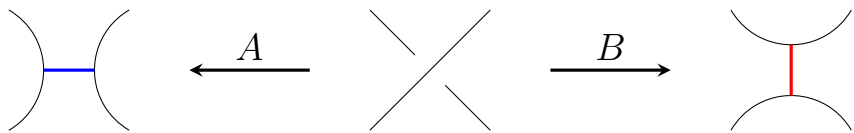

The Turaev surface $F(D)$ of a link diagram $D$ is constructed as follows. The diagram $D$ is embedded on the projection sphere $S^2$. Embed the all-$A$ and all-$B$ states of $D$ in a neighborhood of the projection sphere, but on opposite sides. To construct $F(D)$ we first take a cobordism between the all-$A$ state and the all-$B$ state such that the cobordism consists of bands away from the crossings of $D$ and saddles in neighborhoods of crossings (as in Figure \ref{figure:saddle}). We cap off all boundary components of the cobordism with disks to complete the construction of $F(D)$.
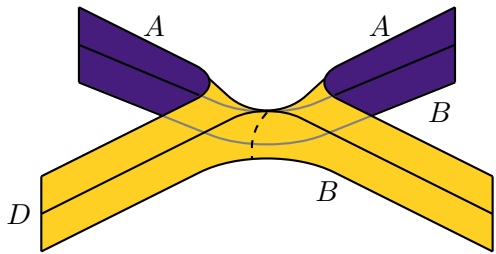
\begin{figure}[h]
$$\begin{tikzpicture}
\begin{scope}[thick]
\draw [rounded corners = 10mm] (0,0) -- (3,1.5) -- (6,0);
\draw (0,0) -- (0,1);
\draw (6,0) -- (6,1);
\draw [rounded corners = 5mm] (0,1) -- (2.5, 2.25) -- (0.5, 3.25);
\draw [rounded corners = 5mm] (6,1) -- (3.5, 2.25) -- (5.5,3.25);
\draw [rounded corners = 5mm] (0,.5) -- (3,2) -- (6,.5);
\draw [rounded corners = 7mm] (2.23, 2.3) -- (3,1.6) -- (3.77,2.3);
\draw (0.5,3.25) -- (0.5, 2.25);
\draw (5.5,3.25) -- (5.5, 2.25);
\end{scope}

\begin{pgfonlayer}{background2}
\fill [lsugold]  [rounded corners = 10 mm] (0,0) -- (3,1.5) -- (6,0) -- (6,1) -- (3,2) -- (0,1); 
\fill [lsugold] (6,0) -- (6,1) -- (3.9,2.05) -- (4,1);
\fill [lsugold] (0,0) -- (0,1) -- (2.1,2.05) -- (2,1);
\fill [lsugold] (2.23,2.28) --(3.77,2.28) -- (3.77,1.5) -- (2.23,1.5);

\fill [white, rounded corners = 7mm] (2.23,2.3) -- (3,1.6) -- (3.77,2.3);
\fill [lsugold] (2,2) -- (2.3,2.21) -- (2.2, 1.5) -- (2,1.5);
\fill [lsugold] (4,2) -- (3.7, 2.21) -- (3.8,1.5) -- (4,1.5);
\end{pgfonlayer}

\begin{pgfonlayer}{background4}
\fill [lsupurple] (.5,3.25) -- (.5,2.25) -- (3,1.25) -- (2.4,2.2);
\fill [rounded corners = 5mm, lsupurple] (0.5,3.25) -- (2.5,2.25) -- (2,2);
\fill [lsupurple] (5.5,3.25) -- (5.5,2.25) -- (3,1.25) -- (3.6,2.2);
\fill [rounded corners = 5mm, lsupurple] (5.5, 3.25) -- (3.5,2.25) -- (4,2);
\end{pgfonlayer}

\draw [thick] (0.5,2.25) -- (1.6,1.81);
\draw [thick] (5.5,2.25) -- (4.4,1.81);
\draw [thick] (0.5,2.75) -- (2.1,2.08);
\draw [thick] (5.5,2.75) -- (3.9,2.08);

\begin{pgfonlayer}{background}
\draw [black!50!white, rounded corners = 8mm, thick] (0.5, 2.25) -- (3,1.25) -- (5.5,2.25);
\draw [black!50!white, rounded corners = 7mm, thick] (2.13,2.07) -- (3,1.7)  -- (3.87,2.07);
\end{pgfonlayer}
\draw [thick, dashed, rounded corners = 2mm] (3,1.85) -- (2.8,1.6) -- (2.8,1.24);
\draw (0,0.5) node[left]{$D$};
\draw (1.5,3) node{$A$};
\draw (4.5,3) node{$A$};
\draw (3.8,.8) node{$B$};
\draw (5.3, 1.85) node{$B$};
\end{tikzpicture}$$
\caption{In a neighborhood of each crossing of $D$ a saddle surface transitions between the all-$A$ and all-$B$ states.}
\label{figure:saddle}
\end{figure}

The genus of the Turaev surface $F(D)$ of $D$ is denoted by $g_T(D)$ and is given by
\begin{equation}
\label{equation:g_T(D)}
g_T(D) = \frac{1}{2}\left(2 + c(D) - s_A(D) - s_B(D)\right),
\end{equation}
where $c(D)$ is the number of crossings in $D$. The {\em Turaev genus} $g_T(L)$ of the link $L$ is given by
$$g_T(L) = \min \{ g_T(D)~|~D~\text{is a diagram of}~L\}.$$

Dasbach, Futer, Kalfagianni, Lin, and Stoltzfus \cite{DFKLS:Determinant} showed how to compute the determinant of a link using a certain graph embedded on the Turaev surface, and they \cite{DFKLS:Jones} also showed that the Jones polynomial of the link is an evaluation of the Bollob\'as-Riordian-Tutte polynomial of that embedded graph. Champanerkar, Kofman, and Stoltzfus \cite{CKS:KhovanovTuraev} showed the support of Khovanov homology gives a lower bound on Turaev genus. A link is {\em adequate} if it has a diagram such that every trace in both the all-$A$ and all-$B$ states intersects two distinct components in the state. Abe \cite{Abe:Adequate} showed that the Khovanov homology bound is exact whenever the link is adequate. In \cite{DasLow:KhovanovTuraev} we gave a model of Khovanov homology based on graphs embedded in the Turaev surface. Lowrance \cite{Low:HFKTuraev} showed that the support of knot Floer homology gives a lower bound on Turaev genus and discussed the relationship between Turaev genus and other link invariants called alternating distances \cite{Low:AltDist}. In \cite{DasLow:TuraevConcordance} we constructed lower bounds on Turaev genus from knot signature, the Ozsv\'ath-Szab\'o $\tau$-invariant, and Rasmussen $s$-invariant. Kalfagianni \cite{Kalfagianni:Adequate} gave a characterization of adequate links in terms of their Turaev genus and colored Jones polynomials.

\subsection{Alternating decompositions}
\label{subsection:altdecomp}

Armond and Lowrance \cite{ArmLow:Turaev} and Kim \cite{Kim:TuraevClassification} studied the Turaev surface via the alternating decompositions of link diagrams of Thistlethwaite \cite{Thistlethwaite:JonesBreadth}. We consider a link diagram $D$ as a $4$-regular graph whose vertices correspond to crossings and where the edges meeting at a vertex are decorated with over/under information. An edge is called non-alternating if both of its endpoints are over-crossings or if both of its endpoints are under-crossings. An {\em alternating decomposition} of $D$ is a pair $(D,\{\gamma_1,\dots,\gamma_k\})$ where $\gamma_1,\dots,\gamma_k$ are simple closed curves in the plane obtained as follows. Each non-alternating edge of $D$ is marked with distinct points. Inside of each face of $D$, the marked points are connected by arcs as in Figure \ref{figure:arcs}. The resulting set of curves is $\{\gamma_1,\dots,\gamma_k\}$.
\begin{figure}[h]
$$\begin{tikzpicture}
\begin{scope}[very thick]
	\draw (.7,0) -- (2.2,0);
	\draw (2.5,0) -- (2.8,0);
	\draw (2.1,-.3) -- (3.6,1.2);
	\draw (3.4,1.2) -- (3.4,2.2);
	\draw (3.4, .8) -- (3.4,.6);
	\draw (3.4, 2.6) -- (3.4,2.8);
	\draw (3.6,2.2) -- (2.6,3.2);
	\draw (2.3,3.5) -- (2.1,3.7);
	\draw (2.7,3.4) -- (.7,3.4);
	\draw (.9,3.3) -- (.1,2.5);
	\draw (1.1,3.5) -- (1.3,3.7);
	\draw (-.1,2.3) -- (-.3,2.1);
	\draw (0,2.7) -- (0,.7);
	\draw (0.1,.9) -- (.9,.1);
	\draw (-.1,1.1) -- (-.3,1.3);
	\draw (1.1,-.1) -- (1.3,-.3);
\end{scope}

\node (1) at (.7,.3){};
\node (2) at (.3,.7){};
\node (3) at (0,1.5) {};
\node (4) at (0,1.9) {};
\node (5) at (.3,2.7) {};
\node (6) at (.7,3.1) {};
\node (7) at (1.5,3.4) {};
\node (8) at (1.9, 3.4) {};
\node (9) at (3.4,1.9) {};
\node (10) at (3.4,1.5) {};
\node (11) at (3.1,.7) {};
\node (12) at (2.7,.3) {};

\begin{scope}[line/.style={shorten >=-0.2cm,shorten <=-0.2cm},thick,red]
\fill (1) circle (.1cm);
\fill (2) circle (.1cm);
\fill (3) circle (.1cm);
\fill (4) circle (.1cm);
\fill (5) circle (.1cm);
\fill (6) circle (.1cm);
\fill (7) circle (.1cm);
\fill (8) circle (.1cm);
\fill (9) circle (.1cm);
\fill (10) circle (.1cm);
\fill (11) circle (.1cm);
\fill (12) circle (.1cm);
\path [bend left, line]   (1) edge (12);
\path [bend left, line]   (3) edge (2);
\path [bend left, line]   (5) edge (4);
\path [bend left, line]   (7) edge (6);
\path [bend left, line]   (9) edge (8);
\path [bend left, line]   (11) edge (10);

\end{scope}
\end{tikzpicture}$$
\caption{Each non-alternating edge is marked with two points. Inside of each face, draw arcs that connect marked points that are adjacent on the boundary but do not lie on the same edge of $D$.}
\label{figure:arcs}
\end{figure}
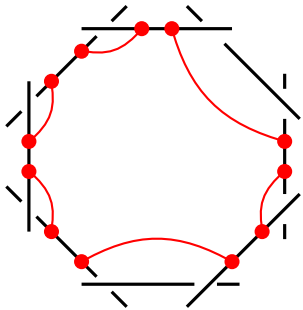

The collection of curves $\{\gamma_1,\dots,\gamma_k\}$ partition the diagram into maximally alternating regions, and these regions are often tangles. This approach can be used to prove the following theorem.
\begin{theorem}[Armond - Lowrance \cite{ArmLow:Turaev}, Kim \cite{Kim:TuraevClassification}]
\label{theorem:genus1}
If $L$ is a non-split link of Turaev genus one, then $L$ has a diagram $D$ obtained by arranging $2k$ proper alternating two-tangles in a cycle, as in Figure \ref{figure:genus1}.
\end{theorem} 

One can also use a collection of previous results to show that all but four prime Turaev genus one knots are hyperbolic. These four knots are the torus knots $T_{3,4}$, $T_{3,5}$ and their mirrors, and also happen to be the non-alternating torus pretzel knots.
\begin{proposition}
If $K$ is a prime knot of Turaev genus one, then $K$ is either hyperbolic or a torus pretzel knot.
\end{proposition}
\begin{proof}
Adams \cite{Adams:ToroidallyAlternating} proved that every prime toroidally alternating knot is either hyperbolic or a torus knot. Since Turaev genus one knots are toroidally alternating, the same holds for them. Abe \cite{Abe:Torus} proved that the only torus knots for which $|s(K) + \sigma(K)| \leq 2$ are $T_{2,2n-1}$, $T_{3,4}$, $T_{3,5}$, and their mirrors. We prove in \cite{DasLow:TuraevConcordance} that $|s(K) +\sigma(K)|\leq 2g_T(K)$. Since $T_{2,2n-1}$ are alternating, it follows that the only torus knots of Turaev genus one are $T_{3,4}$, $T_{3,5}$, and their mirrors. These four knots are the only non-alternating knots that are both torus and pretzel knots by Kawauchi \cite[Theorem~2.3.2]{Kawauchi:Survey}. 
\end{proof}

Non-alternating pretzel links and non-alternating Montesinos links are all Turaev genus one. All non-alternating knots with ten or fewer crossings are Turaev genus one, and most non-alternating knots with twelve or fewer crossings are also Turaev genus one (see \cite{Jablan:Turaev} and Section \ref{section:Jones}). Figure \ref{figure:12n888} shows the mirror of the knot $12n_{888}$ and its alternating decomposition. Since the knot is non-alternating and has an alternating decomposition in the form of Figure \ref{figure:genus1}, its Turaev genus is one.
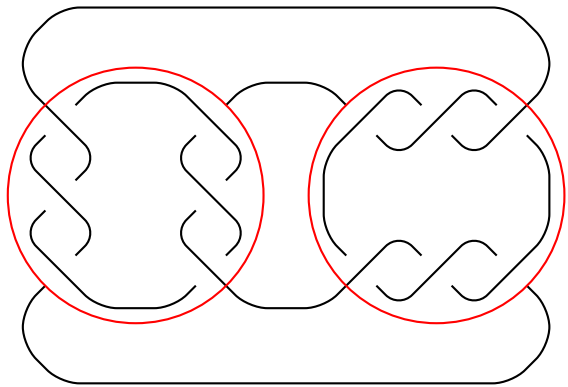
\begin{figure}[h]
$$\begin{tikzpicture}[thick, rounded corners = 2.5mm]

\draw (6.7,2.3) -- (7,2)--(8,3)--(8,3.5)--(7.5,4)-- (1.5,4) -- (1,3.5) -- (1,3) -- (2,2) -- (1.7,1.7);
\draw (1.3,1.3) -- (1,1) -- (2,0) -- (3,0) -- (3.3,.3);
\draw (3.7,.7) -- (4,1) -- (3,2) -- (3.3,2.3);
\draw (1.3,2.3) -- (1,2) -- (2,1) -- (1.7,.7);
\draw (1.7,2.7) -- (2,3) -- (3,3) -- (4,2) -- (3.7,1.7);
\draw (3.3,1.3) -- (3,1) -- (4,0) -- (5,0) -- (6,1) -- (6.3,.7);

\draw (3.7,2.7) -- (4,3) -- (5,3) -- (5.3,2.7);
\draw (5.7,2.3) -- (6,2) -- (7,3) -- (7.3,2.7);
\draw (5.3,.7) -- (5,1) -- (5,2) -- (6,3) -- (6.3,2.7);
\draw (5.7,.3) -- (6,0) -- (7,1) -- (7.3,.7);
\draw (6.7,.3) -- (7,0) --(8,1) -- (8,2) -- (7.7,2.3);
\draw (7.7, .3) -- (8,0) -- (8,-.5) -- (7.5,-1) -- (1.5,-1) -- (1,-.5) -- (1,0) -- (1.3,.3);



\draw [red] (2.5,1.5) circle (1.7cm);
\draw [red] (6.5,1.5) circle (1.7cm);

\end{tikzpicture}$$
\caption{The knot $\overline{12n_{888}}$ and its alternating decomposition.}
\label{figure:12n888}
\end{figure}

\section{Signature}
\label{section:signature}

In this section, we prove Theorems \ref{theorem:mainsignature} and \ref{theorem:signature2}. Using work of Gordon and Litherland \cite{GordonLitherland:Signature}, Thistlethwaite \cite{Thistlethwaite:Adequate}, and Murasugi \cite{Murasugi:Graphs}, the authors previously showed the following theorem. Recall that given a link diagram $D$, the number of components in the all-$A$ and all-$B$ states are given by $s_A(D)$ and $s_B(D)$ respectively. Also, $c_+(D)$ and $c_-(D)$ denote the number of positive and negative crossings in $D$, as in Figure \ref{figure:crossingsign}.
\begin{proposition}[Dasbach, Lowrance - Proposition 5.3 \cite{DasLow:TuraevConcordance}]
\label{proposition:TuraevSig}
Let $L$ be a non-split link with diagram $D$. Then
\begin{equation}
\label{inequality:TuraevSig}
s_A(D) - c_+(D) - 1 \leq \sigma(L) \leq -s_B(D) + c_-(D) +1.
\end{equation}
\end{proposition}
We note that this proposition is stated only for knots in \cite{DasLow:TuraevConcordance}. However, the results that it is based on in \cite{Thistlethwaite:Adequate} and \cite{Murasugi:Graphs} are valid for links of an arbitrary number of components. Moreover, the proof for a link of multiple components is the same as the proof for knots.

The Alexander polynomial of a link is determined by the skein relation
$$\Delta_{L_+}(t) - \Delta_{L_-}(t) = (t^{1/2} - t^{-1/2})\Delta_{L_0}(t).$$
Evaluating the Alexander polynomial of a knot at $t=1$ always yields $1$, since the skein relation becomes $\Delta_{K_+}(1) = \Delta_{K_-}(1)$ and the $\Delta_U(t)=1$ where $U$ is the unknot. Also, the Alexander polynomial is symmetric, i.e.
$$\Delta_L(t) = a_0 + \sum_{i=1}^n a_i(t^i+t^{-i}),$$
for some non-negative integer $n$ and some integer coefficients $a_i$. For a knot $K$, we have
\begin{align*}
\Delta_K(-1) = &\; a_0 + \sum_{i=1}^n (-1)^i 2a_i\\
\equiv &\; a_0 + \sum_{i=1}^n 2a_i \mod 4\\
\equiv &\; \Delta_K(1) \mod 4.
\end{align*}
Therefore, for any knot $K$, we have $\Delta_K(-1)\equiv 1 \mod 4$. Giller \cite{Giller:ConwayCalculus} used this fact to prove the following theorem. 
\begin{theorem}[Giller \cite{Giller:ConwayCalculus}]
\label{theorem:SigAlex}
Suppose that $K$ is a knot with diagram $D$. Then the signature of $K$ can be determined by the following three statements.
\begin{enumerate}
\item If $K$ is the unknot, then $\sigma(K)=0$.
\item If $K_+$ and $K_-$ have diagrams $D_+$ and $D_-$ that differ by a single crossing change where $D_+$ has the positive crossing and $D_-$ has the negative crossing, then
$$\sigma(K_-) - 2\leq \sigma(K_+) \leq \sigma(K_-).$$
\item The Alexander polynomial $\Delta_K(t)$ and the signature $\sigma(K)$ satisfy
$$\operatorname{sign} \Delta_K(-1) = (-1)^{\frac{\sigma(K)}{2}}.$$
\end{enumerate}
\end{theorem}

Proposition \ref{proposition:TuraevSig} and  Theorem \ref{theorem:SigAlex} give us the tools necessary to prove Theorem \ref{theorem:mainsignature}. Theorem \ref{theorem:signature2} then follows from Theorem \ref{theorem:mainsignature}.
\begin{proof}[Proof of Theorem \ref{theorem:mainsignature}]
Let $D$ be a diagram with $g_T(D)=1$. Equation \ref{equation:g_T(D)} implies that $c(D) -s_A(D) - s_B(D) = 0$. Since $c(D)=c_+(D) + c_-(D)$, it follows that 
$s_A(D) - c_+(D) + 1 = -s_B(D) + c_-(D)  + 1$. Therefore, Inequality \ref{inequality:TuraevSig} implies that
$$s_A(D) - c_+(D) -1 \leq \sigma(K) \leq s_A(D) - c_+(D) + 1.$$

Because the signature of a knot is always even, Traczyk's formula (Equation \ref{equation:signaturealternating}) implies that for any alternating knot diagram $D_{\alt}$, the quantity $s_A(D_{\alt}) - c_+(D_{\alt})  -1$ is even. Changing a crossing of a knot diagram changes the number of components in the all-$A$ state by one and changes the number of positive crossings by one. Since $D$ can be obtained from $D_{\alt}$ via a sequence of crossing changes, it follows that $s_A(D)- c_+(D) - 1$ is even.
Therefore $\sigma(K) = s_A(D) - c_+(D) -1$ or $s_A(D) - c_+(D) +1$. Moreover, since $\Delta_K(-1)\equiv 1 \mod 4$, Condition (3) from Theorem \ref{theorem:SigAlex} is equivalent to $\sigma(K) \equiv \det K -1\mod 4$. Therefore if $s_A(D) - c_+(D) -1 \equiv \det(K) -1 \mod 4$, then $\sigma(K) = s_A(D) - c_+(D) -1$ and if $s_A(D) - c_+(D) + 1\equiv \det(K)-1$, then $\sigma(K) = s_A(D) - c_+(D)  + 1$.
\end{proof}

\begin{proof}[Proof of Theorem \ref{theorem:signature2}]
We prove the result in the case where $D$ has the numerator orientation. The proof when $D$ has the denominator orientation is similar. Let $s_A^{\text{int}}(D)$  be the number of components of the all-$A$ state of $D$ that are completely contained within one of the tangles $R_i$. Similarly, let $s_A^{\text{int}}(N(R_i))$ be the number of components of the all-$A$ state of $N(R_i)$ that are completely contained in the tangle $R_i$. The total number of interior components of the all-$A$ state of $D$ is the same as the sum of the number of interior components of the all-$A$ states of $N(R_i)$, i.e.
$$s_A^{\text{int}}(D) = \sum_{i=1}^{2k} s_A^{\text{int}}(N(R_i)).$$
Also, if $i$ is odd, then $s_A(N(R_i)) = s_A^{\text{int}}(N(R_i)) + 1$, and if $i$ is even, then $s_A(N(R_i)) = s_A^{\text{int}}(N(R_i)) + 2$. Furthermore, $s_A(D) = s_A^{\text{int}}(D) + k$. 

Since $D$ has the numerator orientation, a crossing is positive in $N(R_i)$ if and only if it is also positive in $D$. Therefore $c_+(D) = \sum_{i=1}^{2k} c_+(N(R_i))$. Since each $N(R_i)$ is an alternating diagram, we can apply Traczyk's formula (Equation \ref{equation:signaturealternating}) to obtain
\begin{align*}
\sum_{i=1}^{2k} \sigma(N(R_i)) =&\;  \sum_{i=1}^{2k} (s_A(N(R_i)) - c_+(N(R_i)) - 1)\\
 = & \; -c(D) - 2k + \sum_{i=1}^{2k} s_A(N(R_i))\\
 = & \; -c(D) - 2k + \sum_{i=1}^k s_A(N(R_{2i-1})) +  \sum_{i=1}^k s_A(N(R_{2i}))\\
 = & \; -c(D) - 2k + \sum_{i=1}^k (s_A^{\text{int}}(N(R_{2i-1}))+1) + \sum_{i=1}^k (s_A(N(R_{2i})) +2)\\
 = & \; -c(D) +k + \sum_{i=1}^{2k} s_A^{\text{int}}(N(R_i))\\
 = & \; -c(D) + s_A(D).
\end{align*}
Proposition \ref{proposition:TuraevSig} implies that 
$$\sigma(L) = s_A(D) - c(D) \pm 1 = \sum_{i=1}^{2k} \sigma(N(R_i)) \pm 1,$$
as desired.
\end{proof}

Let $R$ be the tangle obtained by connecting the northeast and southeast ends of $R_i$ to the northwest and southwest ends of $R_{i+1}$ respectively for $i=1,\dots, 2k-1$. The numerator closure of $R$ is the diagram of the link $L$ in Figure \ref{figure:genus1}, and the denominator closure of $R$ is $D(R_1)\# \cdots \# D(R_{2k})$. Conway \cite{Conway:Enumeration} proved that
$$\frac{\det N(R)}{\det D(R)} =\left| \sum_{i=1}^{2k} (-1)^i \frac{\det N(R_i)}{\det D(R_i)}\right|.$$
Consequently,
\begin{align}
\begin{split}
\label{eq:det}
 \det L = & \;  \det D(R) \left |\sum_{i=1}^{2k} (-1)^i \frac{\det N(R_i)}{\det D(R_i)}\right|\\
 = & \;  \prod_{i=1}^{2k} \det D(R_i) \left |\sum_{i=1}^{2k} (-1)^i \frac{\det N(R_i)}{\det D(R_i)}\right|\\
 = & \; \left|\sum_{i=1}^{2k}(-1)^i\det D(R_1)\cdots \det D(R_{i-1}) \det N(R_i) \det D(R_{i+1}) \cdots \det D(R_{2k})\right|.
\end{split}
 \end{align}

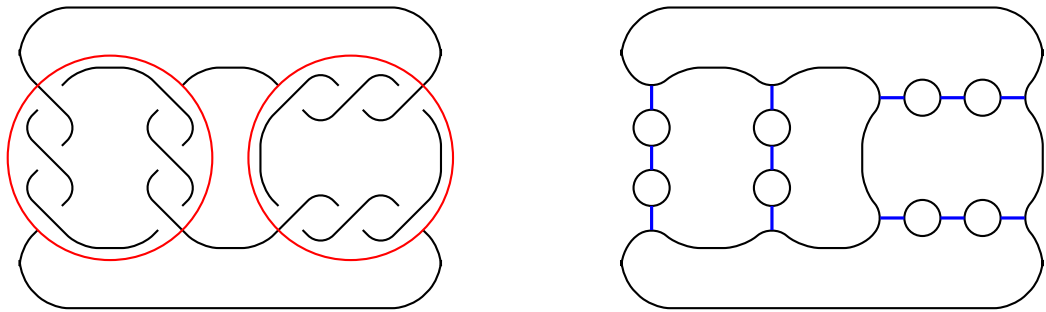
\begin{figure}[h]
$$\begin{tikzpicture}[thick, rounded corners = 2.4mm, scale = .8]

\draw (6.7,2.3) -- (7,2)--(8,3)--(8,3.5)--(7.5,4)-- (1.5,4) -- (1,3.5) -- (1,3) -- (2,2) -- (1.7,1.7);
\draw (1.3,1.3) -- (1,1) -- (2,0) -- (3,0) -- (3.3,.3);
\draw (3.7,.7) -- (4,1) -- (3,2) -- (3.3,2.3);
\draw (1.3,2.3) -- (1,2) -- (2,1) -- (1.7,.7);
\draw (1.7,2.7) -- (2,3) -- (3,3) -- (4,2) -- (3.7,1.7);
\draw (3.3,1.3) -- (3,1) -- (4,0) -- (5,0) -- (6,1) -- (6.3,.7);

\draw (3.7,2.7) -- (4,3) -- (5,3) -- (5.3,2.7);
\draw (5.7,2.3) -- (6,2) -- (7,3) -- (7.3,2.7);
\draw (5.3,.7) -- (5,1) -- (5,2) -- (6,3) -- (6.3,2.7);
\draw (5.7,.3) -- (6,0) -- (7,1) -- (7.3,.7);
\draw (6.7,.3) -- (7,0) --(8,1) -- (8,2) -- (7.7,2.3);
\draw (7.7, .3) -- (8,0) -- (8,-.5) -- (7.5,-1) -- (1.5,-1) -- (1,-.5) -- (1,0) -- (1.3,.3);

\draw [red] (2.5,1.5) circle (1.7cm);
\draw [red] (6.5,1.5) circle (1.7cm);

\begin{scope}[xshift = 10cm]
\draw (4,4) -- (1.5,4) -- (1,3.5) -- (1,3) -- (1.5,2.6) -- (2,3) -- (3,3) -- (3.5,2.6) -- (4,3) -- (5,3) -- (5.4,2.5) -- (5,2) -- (5,1) -- 
(5.4,.5) -- (5,0) -- (4,0) -- (3.5,.4) -- (3,0) -- (2,0) -- (1.5,.4) -- (1,0) -- (1,-.5) -- (1.5,-1) -- (7.5,-1) -- (8,-.5) -- (8,0) -- (7.6,.5) --
(8,1) -- (8,2) -- (7.6,2.5) -- (8,3) -- (8,3.5) -- (7.5,4) -- (4,4);

\draw (1.5,2) circle (.3cm);
\draw (1.5,1) circle (.3cm);

\draw (3.5, 2) circle (.3cm);
\draw (3.5,1) circle (.3cm);

\draw (7,.5) circle (.3cm);
\draw (6,.5) circle (.3cm);

\draw (6,2.5) circle (.3cm);
\draw (7,2.5) circle (.3cm);

\draw [very thick, blue] (1.5,2.3) -- (1.5,2.7);
\draw [very thick, blue] (3.5,2.3) -- (3.5,2.7);
\draw [very thick, blue] (1.5,1.7) -- (1.5,1.3);
\draw [very thick, blue] (3.5,1.7) -- (3.5,1.3);
\draw [very thick, blue] (1.5, .7) -- (1.5,.3);
\draw [very thick, blue] (3.5, .7) -- (3.5,.3);

\draw [very thick, blue] (5.3, 2.5) -- (5.7,2.5);
\draw [very thick, blue] (6.3,2.5) -- (6.7,2.5);
\draw [very thick, blue] (7.3,2.5) -- (7.7,2.5);
\draw [very thick, blue] (5.3, .5) -- (5.7,.5);
\draw [very thick, blue] (6.3,.5) -- (6.7,.5);
\draw [very thick, blue] (7.3,.5) -- (7.7,.5);

\end{scope}

\end{tikzpicture}$$
\caption{The knot $\overline{12n_{888}}$ with its alternating decomposition is on the left, and its all-$A$ state is on the right.}
\label{figure:12n888}
\end{figure}

\begin{example}
Let $K$ be the knot with diagram $D$ as in Figure \ref{figure:12n888}. Then $s_A(D) = 9$, and since every crossing in $D$ is negative, we have $c_+(D)=0$. Theorem \ref{theorem:mainsignature} implies that $\sigma(K) = 8$ or $10$. The numerator closure 
$N(R_1)$ and the denominator closure $D(R_2)$ are $(2,6)$ torus links, while the denominator closure $D(R_1)$ and the numerator closure $N(R_2)$ are the connected sum of two left handed trefoils. Thus 
$$\det N(R_1) = \det D(R_2) = 6~\text{and}~\det D(R_1)=\det N(R_2) = 9.$$
Equation \ref{eq:det}  implies that
$$\det K = |-6\cdot 6 + 9\cdot 9| = 45.$$
Since $45 - 1 \equiv 0 \mod 4$, it follows that $\sigma(K) = 8$.
\end{example}

\section{Jones polynomial}
\label{section:Jones}

In this section, we prove Theorem \ref{theorem:AlmostAlternating} and use it to compute the Turaev genus and dealternating numbers of some knots with 12 or fewer crossings. 

If $s$ is a state of $D$, then the {\em state graph} of $s$ is the graph whose vertices are in one-to-one correspondence with the components of $s$ and whose edges are in one-to-one correspondence with the traces of $s$ (and hence the crossings of $D$). The endpoints of each trace lie on either one or two components of the state $s$, and the edge corresponding to that trace is incident to the vertex or vertices corresponding to those components. If $s$ is the all-$A$ state, then its state graph is called the {\em all-$A$ state graph} of $D$, and is denoted by $G$. Similarly, if $s$ is the all-$B$ state, then its state graph is called the {\em all-$B$ state graph} of $D$, and is denoted by $\overline{G}$. If the diagram $D$ is alternating, then $G$ and $\overline{G}$ are the checkerboard graphs of $D$. The all-$A$ state graph of the diagram in Figure \ref{figure:12n888} is four triangles glued along a common vertex.

Let $D_{\alt}$ be an alternating link diagram with $c=c(D_{\alt})$ crossings. Let $G$ and $\overline{G}$ be the all-$A$ and all-$B$ state graphs. Let $G'$ and $\overline{G}'$ denote the graphs $G$ and $\overline{G}$ where multiple edges are replaced by a single edge. Let $v$ be the number of vertices in $G'$ (or equivalently in $G$), and let $\overline{v}$ be the number of vertices in $\overline{G}'$ (or equivalently $\overline{G}$). Also, let $e$ denote the number of edges of $G'$, and let $\overline{e}$ denote the number of edges of $\overline{G}'$. Dasbach and Lin \cite{DasbachLin:HeadTail} showed that the Kauffman bracket of $D_{\alt}$ can be expressed in the following way.
\begin{theorem}[Dasbach, Lin \cite{DasbachLin:HeadTail}]
\label{theorem:DL}
Suppose that $D_{\alt}$ is an alternating diagram. Then
\begin{align*}
\langle D_{\alt} \rangle = & \; (-1)^{v-1}A^{c+2v-2} + (-1)^{v-2}(e-v+1)A^{c+2v-6}\\
&  \; + \cdots +(-1)^{\overline{v}+2}(\overline{e} - \overline{v} +1)A^{6-c-2\overline{v}} + (-1)^{\overline{v}+1}A^{2-c-2\overline{v}}.
\end{align*}
\end{theorem}

Recall that a link $L$ is {\em almost alternating} if it is non-alternating and has an almost alternating diagram $D$, that is a diagram where one crossing change transforms $D$ into an alternating diagram. Figure \ref{figure:AlmostAlternating} shows a generic almost alternating diagram $D$. Label the four faces of the diagram $D$ incident to the almost alternating crossing by $u_1, u_2, v_1, v_2$. If there is a crossing inside of the alternating tangle $R$ incident to both $u_1$ and $u_2$, then a flype may be applied to the crossing to move it outside of $R$. Then the crossing can be cancelled with the almost alternating crossing via a Reidemeister 2 move, resulting in an alternating diagram. Similarly, if there is a crossing inside of $R$ incident to both $v_1$ and $v_2$, then the diagram can be transformed into an alternating diagram (see the proof of Corollary 4.5 in \cite{Adams:AlmostAlternating}). Therefore, if $L$ is almost alternating, then it has a diagram $D$ as in Figure \ref{figure:AlmostAlternating} where both $N(R)$ and $D(R)$ are reduced alternating diagrams.

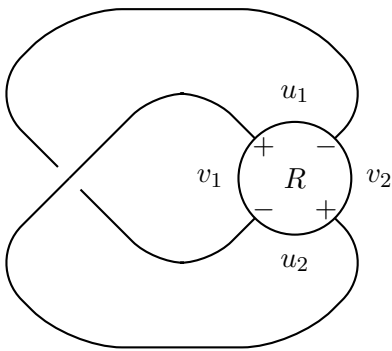
\begin{figure}[h]
$$\begin{tikzpicture}[scale=.75]

\draw[thick, rounded corners=4mm] (4,0) -- (2.5,1.5) -- (1.5,1.5) -- (0,0);
\draw[thick, rounded corners = 4mm] (4,0) -- (2.5,-1.5) -- (1.5,-1.5) -- (.2,-.2);
\draw[thick, rounded corners = 7mm] (4,0) -- (5.5,1.5) -- (4,3) -- (0,3) -- (-1.5,1.5) -- (-.2,.2);
\draw[thick, rounded corners = 7mm] (0,0) -- (-1.5,-1.5) -- (0,-3) -- (4,-3) -- (5.5,-1.5) -- (4,0);

\fill[white] (4,0) circle (1cm);
\draw[thick] (4,0) circle (1cm);
\draw (4,0) node{$R$};
\draw (3.45,.55) node{$+$};
\draw (4.55,.55) node{$-$};
\draw (4.55,-.55) node{$+$};
\draw (3.45,-.55)node{$-$};

\draw (4,1.5) node{$u_1$};
\draw (4,-1.5) node{$u_2$};
\draw (2.5,0) node{$v_1$};
\draw (5.5,0) node{$v_2$};

\end{tikzpicture}$$
\caption{An almost alternating diagram. The two-tangle $R$ is alternating.}
\label{figure:AlmostAlternating}
\end{figure}

Two faces $f_1$ and $f_2$ of a link diagram are {\em adjacent} if there exists a crossing incident to $f_1$ and $f_2$.  Let $\adj(u_1,u_2)$ be the number of faces of $D$ that are contained in $R$ and are adjacent to both $u_1$ and $u_2$, and let $\adj(v_1,v_2)$ be the number of faces of $D$ that are contained in $R$ and are adjacent to both $v_1$ and $v_2$. See Figure \ref{figure:adjacent}.

\begin{figure}[h]
$$\begin{tikzpicture}[thick,scale = .8, rounded corners = 3mm]

\draw (3,0) node{$u_2$};
\draw (3,4) node{$u_1$};

\fill[black!30!white] (1,3) -- (0.5,2.5) -- (.5,1.5) -- (1,1) -- (2,2) -- (1,3);
\fill[black!30!white] (2,2) -- (3,3) -- (3.5,2.5) -- (3.5,1.5) -- (3,1) -- (2,2);
\fill[black!30!white] (5,1) -- (5.5,1.5) -- (5.5,2.5) -- (5,3) -- (4.5,2.5) -- (4.5,1.5) -- (5,1);
\fill[black!30!white,rounded corners = 0mm] (.8,1.2) -- (1,1) -- (2,2) -- (1.8,2.2);
\fill[black!30!white,rounded corners = 0mm] (2.8,2.8) -- (3,3) -- (3.2,2.8);
\fill[black!30!white,rounded corners = 0mm] (2.8, 1.2) -- (3,1) -- (3.2,1.2);
\fill[black!30!white,rounded corners = 0mm] (4.8,2.8) -- (5,3) -- (5.2,2.8);
\draw (5.1, 3.1) -- (6,4) -- (6,5) -- (0,5) -- (0,4) -- (1.9,2.1);
\draw (2.9,2.9) -- (1.1,1.1);

\draw (.9,2.9) -- (0.5,2.5) --(.5,1.5)  -- (1.5,.5) -- (2.5,.5) -- (2.9,.9);
\draw (3.1,1.1) -- (3.5,1.5) -- (3.5,2.5) -- (2.5,3.5) -- (1.5,3.5) -- (1.1,3.1);
\draw (2.1,1.9) -- (3.5,.5) -- (4.5,.5) -- (4.9,.9);
\draw (5.1,1.1) -- (5.5,1.5) -- (5.5,2.5) -- (4.5,3.5) -- (3.5,3.5) -- (3.1,3.1);
\draw (.9,.9) -- (0,0) -- (0,-1) -- (6,-1) --(6,0) -- (4.5,1.5) -- (4.5,2.5) -- (4.9,2.9);

\begin{scope}[xshift = 9cm]

\draw (-.5,2) node{$v_1$};
\draw (6.5,2) node{$v_2$};

\draw (.9,.9) -- (0,0) -- (-1,0) -- (-1,4) -- (0,4) -- (1.9,2.1);
\draw (2.9,2.9) -- (1.1,1.1);

\draw (.9,2.9) -- (0.5,2.5) --(.5,1.5)  -- (1.5,.5) -- (2.5,.5) -- (2.9,.9);
\draw (3.1,1.1) -- (3.5,1.5) -- (3.5,2.5) -- (2.5,3.5) -- (1.5,3.5) -- (1.1,3.1);
\draw (2.1,1.9) -- (3.5,.5) -- (4.5,.5) -- (4.9,.9);
\draw (5.1,1.1) -- (5.5,1.5) -- (5.5,2.5) -- (4.5,3.5) -- (3.5,3.5) -- (3.1,3.1);
\draw (5.1,3.1) -- (6,4) -- (7,4) -- (7,0) -- (6,0) -- (4.5,1.5) -- (4.5,2.5) -- (4.9,2.9);

\end{scope}

\end{tikzpicture}$$

\caption{The numerator and denominator closures of an alternating tangle $R$. In this example, $\adj(u_1,u_2)=3$ while $\adj(v_1,v_2)=0$. Faces that are adjacent to both $u_1$ and $u_2$ are shaded.}
\label{figure:adjacent}
\end{figure}
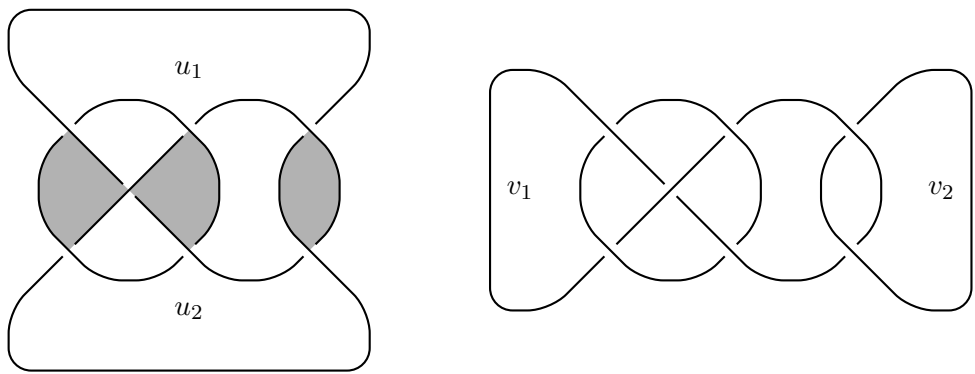

The following lemma shows that the first and last coefficients of the Kauffman bracket of an almost alternating diagram can be expressed in terms of $\adj(u_1,u_2)$ and $\adj(v_1,v_2)$.
\begin{lemma}
\label{lemma:penultimate}
Let $D$ be an almost alternating diagram as in Figure \ref{figure:AlmostAlternating}, and assume that both $N(R)$ and $D(R)$ are reduced alternating diagrams. Then for some integers $p$ and $k$,
$$\langle D \rangle = \sum_{i=0}^k \alpha_i A^{p-4i},$$
where $\alpha_0 = \pm(1-\adj(u_1,u_2))$ and $\alpha_k = \pm(1-\adj(v_1,v_2))$.
\end{lemma}
\begin{proof}
Let $G_N, G_D$ be the all-$A$ state graphs of $N(R)$ and $D(R)$ respectively, and let $\overline{G}_N,$ and $\overline{G}_D$ be the all-$B$ state graphs of $N(R)$ and $D(R)$. Let $G_N'$, $G_D'$, $\overline{G}_N'$, and $\overline{G}_D'$ be these graphs after all multiple edges are replaced by a single edge. Let $v_N$ and $\overline{v}_N$ be the number of vertices of $G_N$ and $\overline{G}_N$ respectively. Let $e_N$ and $\overline{e}_N$ be the number of edges in $G_N'$ and $\overline{G}_N'$. Similarly define $v_D$, $\overline{v}_D$, $e_D$ and $\overline{e}_D$ using $D(R)$ in place of $N(R)$.

Since the graph $G_D$ is obtained from the graph $G_N$ by identifying the vertices corresponding to faces $u_1$ and $u_2$, we have $v_N = v_D +1$. Similarly, $\overline{v}_D = \overline{v}_N + 1$. The graphs $G_N$, $\overline{G}_N$, $G_D$, and $\overline{G}_D$ all have the same number of edges, the number of crossings in $R$. Suppose that $u_3$ is a face adjacent to $u_1$ and $u_2$. Let $e_1$ and $e_2$ be the corresponding edges in $G_N$. The edges $e_1$ and $e_2$ do not have the same endpoints in $G_N$, but since $u_1$ and $u_2$ are identified together to form $G_D$, the corresponding edges have the same endpoints in $G_D$. Therefore $e_N = e_D+\adj(u_1,u_2)$. Similarly, $\overline{e}_D= \overline{e}_N + \adj(v_1,v_2)$.

The Kauffman bracket of $D$ is given by $\langle D \rangle = A\langle D(R) \rangle + A^{-1}\langle N(R) \rangle$. Theorem \ref{theorem:DL} implies
\begin{align*}
A \langle D(R) \rangle = & \; (-1)^{v_D-1}A^{c+2v_D-1} + (-1)^{v_D-2}(e_D-v_D+1)A^{c+2v_D-5}\\
& \; + \cdots +(-1)^{\overline{v}_D+2}(\overline{e}_D - \overline{v}_D +1)A^{7-c-2\overline{v}_D} + (-1)^{\overline{v}_D+1}A^{3-c-2\overline{v}_D}~\text{and}\\
A^{-1}\langle N(R) \rangle =& \; (-1)^{v_N-1}A^{c+2v_N-3} + (-1)^{v_N-2}(e_N-v_N+1)A^{c+2v_N-7}\\
& \; + \cdots +(-1)^{\overline{v}_N+2}(\overline{e}_N - \overline{v}_N +1)A^{5-c-2\overline{v}_N} + (-1)^{\overline{v}_N+1}A^{1-c-2\overline{v}_N}\\
=& \; (-1)^{v_D} A^{c+2v_D-1} + (-1)^{v_D-1}(e_N -v_D+2)A^{c+2v_D-5}\\
& \; + \cdots + (-1)^{\overline{v}_D+1}(\overline{e}_N - \overline{v}_D)A^{7-c-2\overline{v}_D} + (-1)^{\overline{v}_D}A^{3-c-2\overline{v}_D}.
\end{align*}
Therefore, both of the coefficients of $A^{c+2v_D-1}$ and $A^{3-c-\overline{v}_D}$ in $\langle D \rangle$ are zero. Hence the greatest power of $A$ that potentially has nonzero coefficient is $A^{c+2v_D-5}$, and the least power of $A$ that potentially has nonzero coefficient is $A^{7-c-2\overline{v}_D}$. The coefficient of $A^{c+2v_D-5}$ is 
$$(-1)^{v_D}(e_D - e_N + 1) = (-1)^{v_D}(1 - \adj(u_1,u_2)).$$
Similarly, the coefficient of $A^{3-c-2\overline{v}_D}$ is 
$$(-1)^{\overline{v}_D-1}(\overline{e}_N - \overline{e}_D + 1) = (-1)^{\overline{v}_D-1}(1 - \adj(v_1,v_2)),$$
giving us the desired result.
\end{proof}

\begin{proof}[Proof of Theorem \ref{theorem:AlmostAlternating}]
Let $D$ be an almost alternating diagram of $L$ with the fewest number of crossings among all almost alternating diagrams of $L$. If either $N(R)$ or $D(R)$ is not reduced, then either there exists an almost alternating diagram of $L$ with fewer crossings or $L$ is an alternating link. If $L$ is alternating, then both $|a_m|$ and $|a_M|$ are 1 by a result of Kauffman \cite{Kauffman:Bracket}. 

Suppose that both $N(R)$ and $D(R)$ are reduced. It suffices to show that the trailing or leading coefficient of $\langle D \rangle$ is $\pm 1$. By Lemma \ref{lemma:penultimate}, if either $\adj(u_1,u_2)$ or $\adj(v_1,v_2)$ is $0$ or $2$, then the result is shown. Let $\Gamma$ and $\Gamma^*$ be the checkerboard graphs of $D$ such that $u_1$ and $u_2$ are vertices in $\Gamma$ and $v_1$ and $v_2$ are vertices in $\Gamma^*$. Suppose that $e_1$ and $e_2$ are the only two edges in a path between $u_1$ and $u_2$. Then any path between $v_1$ and $v_2$ must contain either the edge dual to $e_1$ or the edge dual to $e_2$. Hence the number of disjoint paths between $u_1$ and $u_2$ is a lower bound for the length of the shortest path between $v_1$ and $v_2$. Therefore if $\adj(u_1,u_2)\geq 3$, then $\adj(v_1,v_2) = 0$, and similarly if $\adj(v_1,v_2)\geq 3$, then $\adj(u_1,u_2)=0$.

Suppose $\adj(u_1,u_2)\neq 1$. Either $\adj(u_1,u_2) =0$ or $2$ and $|1-\adj(u_1,u_2)|=1$, or $\adj(u_1,u_2)\geq 3$ and $|1-\adj(v_1,v_2)| = |1 - 0| = 1$. By a similar argument, if $\adj(v_1,v_2)\neq 1$, then at least one of $|1-\adj(u_1,u_2)|$ or $|1-\adj(v_1,v_2)|$ is one. Thus the only case left to consider is $\adj(u_1,u_2)=\adj(v_1,v_2)=1$. If $\adj(u_1,u_2)=\adj(v_1,v_2)=1$, then $D$ has diagram as in Figure \ref{figure:adj1} where $R_1$, $R_2$ and $R_3$ are alternating tangles except $R_2$ and $R_3$ are allowed to have no crossings. Furthermore, if $\adj(u_1,u_2) = \adj(v_1,v_2) =1$, then $L$ has an almost alternating diagram with two fewer crossings than $D$ (as depicted in Figure \ref{figure:AAIsotope}), contradicting the minimality of $D$. Therefore, either $\adj(u_1,u_2)\neq 1$ or $\adj(v_1,v_2)\neq 1$, and the result is proven for almost alternating links.

If $L$ is a link with $g_T(L)=1$, then \cite{ArmLow:Turaev} implies that $L$ is mutant to an almost alternating link $L'$. Since mutation does not change the Jones polynomial, it follows that $V_L(t)=V_{L'}(t)$, and the result holds.
\begin{figure}[h]
$$\begin{tikzpicture}[scale=.8]

\draw (1.5,1.2) node{$u_1$};
\draw (-.6,-3.7) node{$u_2$};
\draw (-1,-1.7) node{$v_1$};
\draw (4.5,-1.7) node{$v_2$};

\draw[thick, rounded corners = 3mm] (-.4,-.7) -- (-.4,-1.3) -- (.4,-2.1) -- (.4,-3);
\draw[thick, rounded corners = 3mm] (.4,-.7) -- (.4,-1.3) -- (.1,-1.6);
\draw[thick, rounded corners = 3mm] (-.4,-3) -- (-.4, -2.1) -- (-.1,-1.8);

\draw[thick, rounded corners = 3mm] (.7,.4) -- (1.3,.4) -- (2.1,-.4) -- (3,-.4);
\draw[thick, rounded corners = 3mm] (.7,-.4) -- (1.3, -.4) -- (1.6,-.1);
\draw[thick, rounded corners = 3mm] (3,.4) -- (2.1,.4) -- (1.8,.1);

\draw[thick] (-.5,-2.1) ellipse (2cm and 2.1cm);

\draw[thick, rounded corners = 2mm] (0,-3) -- (.4,-3.4) -- (.7,-3.6);
\draw[thick] (3.2,0) arc (45:-120:1.9 cm and 2.4cm);
\draw[thick] (-3,.9) arc (180:-30:3.45cm and 1cm);
\draw[thick] (-3,.9) arc (180:210: 4cm);
\draw[thick] (-.1,-3.2) arc(-90:-160:2.3cm);

\fill[white] (0,0) circle (1cm);
\draw[thick] (0,0) circle (1cm);
\draw (0,0) node{$R_1$};
\draw (-1,-.1) node[right]{\tiny{+}};
\draw (-.35,-1) node[above]{\tiny{-}};
\draw (.3,-1) node[above]{\tiny{+}};
\draw (.65,-.55) node{\tiny{-}};
\draw (.75,-.3) node{\tiny{+}};
\draw (.75,.35) node{\tiny{-}};

\fill[white](3,0) circle (.6 cm);
\draw[thick] (3,0) circle (.6cm);
\draw (3,0) node {$R_2$};
\draw (2.4,-.3) node[right]{\tiny{-}};
\draw (2.4,.3) node[right]{\tiny{+}};
\draw (3.6,-.3) node[left]{\tiny{+}};
\draw (3.6,.3) node[left]{\tiny{-}};

\fill[white](0,-3) circle (.6cm);
\draw[thick] (0,-3) circle (.6cm);
\draw (0,-3) node{$R_3$};
\draw (-.3,-2.4) node[below]{\tiny{+}};
\draw (.3,-2.4) node[below]{\tiny{-}};
\draw (-.3, -3.5) node[above]{\tiny{-}};
\draw (.3, -3.6) node[above]{\tiny{+}};

\end{tikzpicture}$$
\caption{If $\adj(u_1,u_2) = \adj(v_1,v_2)=1$, then $D$ has the above diagram.}
\label{figure:adj1}
\end{figure}
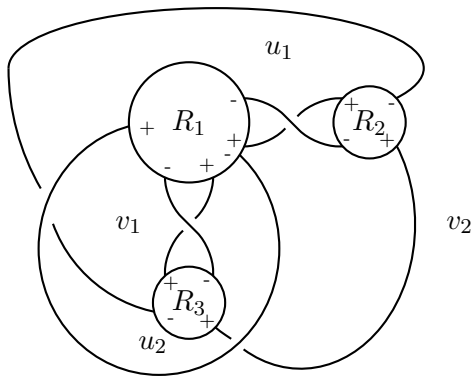

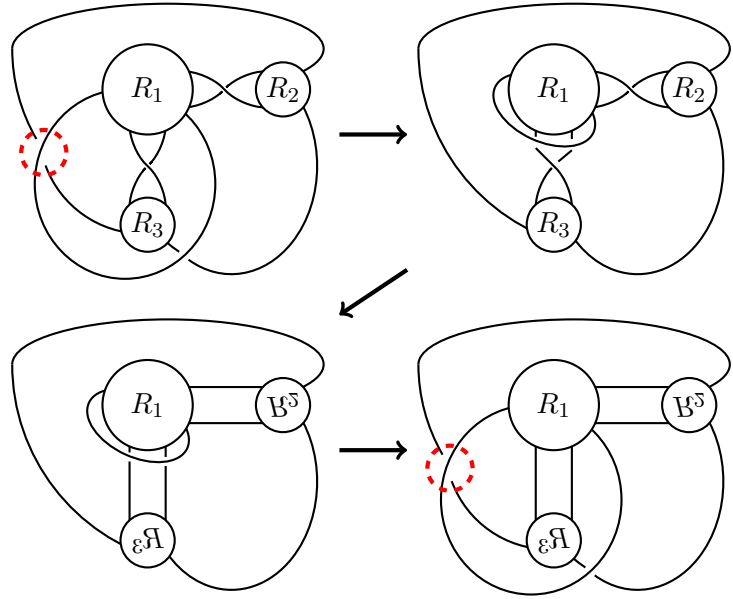
\begin{figure}[h]
$$\begin{tikzpicture}[scale = .6]

\draw[thick, rounded corners = 3mm] (-.4,-.7) -- (-.4,-1.3) -- (.4,-2.1) -- (.4,-3);
\draw[thick, rounded corners = 3mm] (.4,-.7) -- (.4,-1.3) -- (.1,-1.6);
\draw[thick, rounded corners = 3mm] (-.4,-3) -- (-.4, -2.1) -- (-.1,-1.8);

\draw[thick, rounded corners = 3mm] (.7,.4) -- (1.3,.4) -- (2.1,-.4) -- (3,-.4);
\draw[thick, rounded corners = 3mm] (.7,-.4) -- (1.3, -.4) -- (1.6,-.1);
\draw[thick, rounded corners = 3mm] (3,.4) -- (2.1,.4) -- (1.8,.1);

\draw[thick] (-.5,-2.1) ellipse (2cm and 2.1cm);

\draw[thick, rounded corners = 2mm] (0,-3) -- (.4,-3.4) -- (.7,-3.6);
\draw[thick] (3.2,0) arc (45:-120:1.9 cm and 2.4cm);
\draw[thick] (-3,.9) arc (180:-30:3.45cm and 1cm);
\draw[thick] (-3,.9) arc (180:210: 4cm);
\draw[thick] (-.1,-3.2) arc(-90:-160:2.3cm);

\fill[white] (0,0) circle (1cm);
\draw[thick] (0,0) circle (1cm);
\draw (0,0) node{$R_1$};

\fill[white](3,0) circle (.6 cm);
\draw[thick] (3,0) circle (.6cm);
\draw (3,0) node {$R_2$};

\fill[white](0,-3) circle (.6cm);
\draw[thick] (0,-3) circle (.6cm);
\draw (0,-3) node{$R_3$};

\draw[ultra thick, red, dashed] (-2.3,-1.4) circle (.5cm);

\begin{scope}[xshift = 9cm]
\draw[thick] (-.4,-.7) -- (-.4,-1.05);
\draw[thick] (.4,-.7) -- (.4,-1.1);
\draw[thick, rounded corners = 3mm] (-.4,-1.3) -- (.4,-2.1) -- (.4,-3);
\draw[thick, rounded corners = 3mm]  (.4,-1.35) -- (.1,-1.6);
\draw[thick, rounded corners = 3mm] (-.4,-3) -- (-.4, -2.1) -- (-.1,-1.8);

\draw[thick, rounded corners = 3mm] (.7,.4) -- (1.3,.4) -- (2.1,-.4) -- (3,-.4);
\draw[thick, rounded corners = 3mm] (.7,-.4) -- (1.3, -.4) -- (1.6,-.1);
\draw[thick, rounded corners = 3mm] (3,.4) -- (2.1,.4) -- (1.8,.1);

\draw[thick,rotate = -25] (0,-.5) ellipse (1.2cm and .7cm);

\draw[thick] (3.2,0) arc (45:-150:1.9 cm and 2.4cm);
\draw[thick] (-3,.9) arc (180:-30:3.45cm and 1cm);
\draw[thick] (-3,.9) arc (180:255: 4.5cm);

\fill[white] (0,0) circle (1cm);
\draw[thick] (0,0) circle (1cm);
\draw (0,0) node{$R_1$};

\fill[white](3,0) circle (.6 cm);
\draw[thick] (3,0) circle (.6cm);
\draw (3,0) node {$R_2$};

\fill[white](0,-3) circle (.6cm);
\draw[thick] (0,-3) circle (.6cm);
\draw (0,-3) node{$R_3$};

\end{scope}

\begin{scope}[yshift = -7cm]

\draw[thick] (-.4,-.7) -- (-.4,-1.05);
\draw[thick] (-.4,-1.25) --  (-.4,-3);
\draw[thick] (.4,-.7) -- (.4,-1.15);
\draw[thick] (.4,-1.35) -- (.4,-3);

\draw[thick] (.7,.4) -- (3,.4);
\draw[thick] (.7,-.4) -- (3,-.4);

\draw[thick,rotate = -25] (0,-.5) ellipse (1.2cm and .7cm);

\draw[thick] (3.2,0) arc (45:-150:1.9 cm and 2.4cm);
\draw[thick] (-3,.9) arc (180:-30:3.45cm and 1cm);
\draw[thick] (-3,.9) arc (180:255: 4.5cm);

\fill[white] (0,0) circle (1cm);
\draw[thick] (0,0) circle (1cm);
\draw (0,0) node{$R_1$};

\fill[white](3,0) circle (.6 cm);
\draw[thick] (3,0) circle (.6cm);
\draw (3,0) node {$\scalebox{1}[-1]{$R_2$}$};

\fill[white](0,-3) circle (.6cm);
\draw[thick] (0,-3) circle (.6cm);
\draw (0,-3) node{$\scalebox{-1}[1]{$R_3$}$};

\end{scope}

\begin{scope}[xshift = 9cm, yshift = -7cm]

\draw[thick] (-.4,-.7) -- (-.4,-3);
\draw[thick] (.4,-.7) -- (.4,-3);

\draw[thick] (.7,.4) -- (3,.4);
\draw[thick] (.7,-.4) -- (3,-.4);

\draw[thick] (-.5,-2.1) ellipse (2cm and 2.1cm);

\draw[thick, rounded corners = 2mm] (0,-3) -- (.4,-3.4) -- (.7,-3.6);
\draw[thick] (3.2,0) arc (45:-120:1.9 cm and 2.4cm);
\draw[thick] (-3,.9) arc (180:-30:3.45cm and 1cm);
\draw[thick] (-3,.9) arc (180:210: 4cm);
\draw[thick] (-.1,-3.2) arc(-90:-160:2.3cm);

\fill[white] (0,0) circle (1cm);
\draw[thick] (0,0) circle (1cm);
\draw (0,0) node{$R_1$};

\fill[white](3,0) circle (.6 cm);
\draw[thick] (3,0) circle (.6cm);
\draw (3,0) node {$\scalebox{1}[-1]{$R_2$}$};

\fill[white](0,-3) circle (.6cm);
\draw[thick] (0,-3) circle (.6cm);
\draw (0,-3) node{$\scalebox{-1}[1]{$R_3$}$};

\draw[ultra thick, red, dashed] (-2.3,-1.4) circle (.5cm);
\end{scope}

\draw[ultra thick, ->] (4.25,-1) -- (5.75,-1);
\draw[ultra thick, ->] (4.25,-8) -- (5.75,-8);
\draw[ultra thick, ->] (5.75, -4) -- (4.25, -5);

\end{tikzpicture}$$
\caption{If $\adj(u_1,u_2) = \adj(v_1,v_2)=1$, then $D$ is isotopic to another almost alternating $D'$ with two fewer crossings. The crossings contained in the dashed red circles are the almost alternating crossings.}
\label{figure:AAIsotope}
\end{figure}

\end{proof}

Adams et. al. \cite{Adams:AlmostAlternating} extended the notion of almost alternating as follows. The {\em dealternating number} $\dalt(D)$ of a link diagram $D$ is the minimum number of crossing changes necessary to transform the diagram $D$ into an alternating diagram. The {\em dealternating number} $\dalt(L)$ of the link $L$ is the minimum of $\dalt(D)$ over all diagrams $D$ of $L$. A link $L$ is almost alternating if and only if $\dalt(L)=1$. Theorem \ref{theorem:AlmostAlternating} implies that if the first and last coefficients $a_m$ and $a_M$ of the Jones polynomial of $L$ are both two or greater in absolute value, then $g_T(L)\geq 2$ and $\dalt(L)\geq 2$.

\begin{example}
\label{example:11n95}
The knot $K=11n_{95}$ in Figure \ref{figure:11n95} has Jones polynomial
$$V_K(t) =  2t^2 - 3t^3 + 5t^4 - 6t^5 + 6t^6 - 5t^7 + 4t^8 - 2t^9.$$
Theorem \ref{theorem:AlmostAlternating} implies that $g_T(11n_{95})\geq 2$ and $\dalt(11n_{95})\geq 2$. Figure \ref{figure:11n95} gives diagrams of $11n_{95}$ of Turaev genus and dealternating number two. Thus $g_T(11n_{95})=\dalt(11n_{95})=2$.
\end{example}

Kawauchi \cite{Kawauchi:Alternation} defined the {\em alternation number} $\alt(L)$ of a link $L$ to be the Gordian distance from $L$ to the set of alternating links. In other words, for a link diagram $D$, define $\alt(D)$ to be the minimum number of crossings changes necessary to transform $D$ into a (possibly non-alternating) diagram of an alternating link. Then define $\alt(L)$ to be the minimum $\alt(D)$ over all diagrams $D$ of $L$. Figure \ref{figure:11n95} shows that Theorem \ref{theorem:AlmostAlternating} does not extend to alternation number one links. If the crossing marked in the upper left diagram in Figure \ref{figure:11n95} is changed, then the resulting diagram is a trefoil. Thus $\alt(11n_{95}) =1$.

\begin{figure}[h]
$$\begin{tikzpicture}[scale = .8, thick]

\begin{scope}[rounded corners = 2mm]
\draw (4,2.1) -- (4,2.9);
\draw (3.1,3) -- (4.5,3) -- (5.5,2) -- (5.5,1.5)  -- (5,1) -- (3.6,1);
\draw (4.9,2.4) -- (4.5,2) -- (2.5,2) -- (2.5,1.6);
\draw (3,2.1) -- (3,3.5) -- (3.4,3.9);
\draw (4,3.1) -- (4,3.5) -- (3,4.5) -- (1.5,4.5) -- (.5,3.5) -- (.5,3) -- (1.4,3);
\draw (3.6,4.1) -- (4,4.5) -- (4.5,4.5)-- (5.5,3.5)--(5.5,3)--(5.1,2.6);
\draw (2.9,3) -- (1.6,3);
\draw (1.5,1.6) -- (1.5,3.5) -- (1.1,3.9);
\draw (3,1.9) -- (3,1.5) -- (0,1.5) -- (0,4.5) -- (.5,4.5) -- (.9,4.1);
\draw (1.5,1.4) -- (1.5,.5) -- (3,.5) -- (4,1.5) -- (4,1.9);
\draw (2.5,1.4) -- (2.5,1) -- (3.35,1);

\draw[very thick, red] (1,4) circle (.25cm);
\end{scope}

\begin{scope}[xshift = 8cm, rounded corners = 2mm]
\draw (3.1,3) -- (3.5,3) -- (5,1.5) -- (5,1) -- (3.6,1);
\draw (3.9,2.4) -- (3.5,2) -- (2.5,2) -- (2.5,1.6);
\draw (3,2.1) -- (3,3.5) -- (3.4,3.9);
\draw (4.4,3.1) -- (4,3.5) -- (3,4.5) -- (1.5,4.5) -- (.5,3.5) -- (.5,3) -- (1.4,3);
\draw (3.6,4.1) -- (4,4.5) -- (5,3.5) -- (4.1,2.6);
\draw (2.9,3) -- (1.6,3);
\draw (1.5,1.6) -- (1.5,3.5) -- (1.1,3.9);
\draw (3,1.9) -- (3,1.5) -- (0,1.5) -- (0,4.5) -- (.5,4.5) -- (.9,4.1);
\draw (1.5,1.4) -- (1.5,.5) -- (3,.5) -- (4,1.5) -- (4.4,1.9);
\draw (2.5,1.4) -- (2.5,1) -- (3.35,1);
\draw (4.6,2.9) -- (5,2.5) -- (4.6,2.1);
\draw [very thick, red] (2.1,3.2) -- (2.1,.8) -- (2.6,1) -- (3.2,1.3) -- (4.5,2.5) -- (5.2,2.5) -- (5.2,3.5) --  (3.9,4.8) -- (2.1,4.8) -- (2.1,3.1);

\end{scope}

\begin{scope}[xshift = 4cm, yshift = -5.5cm,rounded corners=2mm]
\draw (3.1,3) -- (3.5,3) -- (5,1.5) -- (5,1) -- (3.6,1);
\draw (3.9,2.4) -- (3.5,2) -- (2.5,2) -- (2.5,1.6);
\draw (3,2.1) -- (3,3.5) -- (3.4,3.9);
\draw (1.4,3) -- (.5,3) -- (.5,3.5) -- (1.5,4.5) -- (2,4) -- (2, 3.1);
\draw (2,2.9) -- (2,1.6);
\draw (2,1.4) -- (2,.8) --  (2.5,1);
\draw (2.7,1.1) -- (3.2,1.3) -- (4.15,2.15);
\draw (4.35,2.35) -- (4.5,2.5) -- (4.8,2.5);
\draw (4.4,3.1) -- (3,4.5) -- (3.5,5) -- (4,5) -- (5.5,3.5) -- (5.5,2.5) -- (5,2.5);
\draw (3.6,4.1) -- (4,4.5) -- (5,3.5) -- (4.1,2.6);
\draw (2.9,3) -- (1.6,3);
\draw (1.5,1.6) -- (1.5,3.5) -- (1.1,3.9);
\draw (3,1.9) -- (3,1.5) -- (0,1.5) -- (0,4.5) -- (.5,4.5) -- (.9,4.1);
\draw (1.5,1.4) -- (1.5,.5) -- (3,.5) -- (4,1.5) -- (4.4,1.9);
\draw (2.5,1.4) -- (2.5,1) -- (3.35,1);
\draw (4.6,2.9) -- (5,2.5) -- (4.6,2.1);

\draw [very thick, red] (2,1.5) circle (.25cm);
\draw [very thick, red] (4.25,2.25) circle (.25cm);

\end{scope}

\end{tikzpicture}$$
\caption{On the upper left is the standard diagram of $11n_{95}$.  The Turaev surface of this diagram has genus three, and if the encircled crossing is changed, then the resulting knot is a trefoil. Performing a Reidemeister 3 move yields the diagram on the upper right, whose Turaev surface has genus two. If a strand of $11n_{95}$ is pulled beneath the encircled alternating tangle in the upper right, then the resulting diagram is shown on the bottom. This diagram has dealternating number two. }
\label{figure:11n95}
\end{figure}
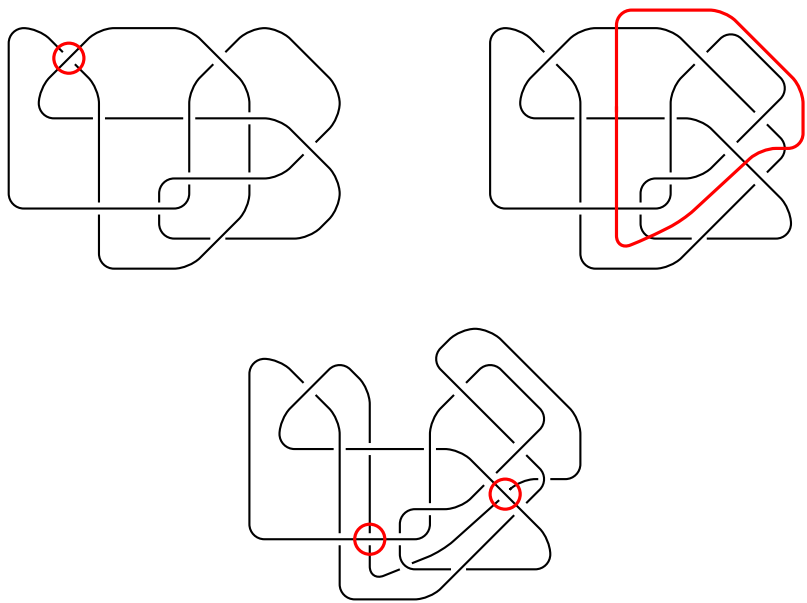

Jablan \cite{Jablan:Turaev} (together with unpublished work of Joshua Howie) showed that all knots with twelve or fewer crossings have Turaev genus and dealternating number at most two. For knots with eleven crossings, all but $11n_{95}$ and $11n_{118}$ are known to be Turaev genus one and almost alternating. Example \ref{example:11n95} shows that $g_T(11n_{95})=\dalt(11n_{95})=2$.  Among all knots with 12 crossings, there are 35 whose Turaev genus and dealternating number are unknown. Theorem \ref{theorem:AlmostAlternating} implies that the eleven knots in Table \ref{table:12KnotsJones} have Turaev genus and dealternating number two. 

\begin{table}[h]
\begin{tabular}{| l | l | }
\hline
Name & $V_K(t)$\\
\hline
\hline
$12n_{253}$ & $-2t^{-8}+ 4t^{-7}-7t^{-6}+ 9t^{-5}-9t^{-4}+ 10t^{-3}-7t^{-2}+ 5t^{-1}-2$\\
\hline
$12n_{254}$ & $3t^2-5t^3+ 9t^4-11t^5+ 11t^6-11t^7+ 8t^8-5t^9+ 2t^{10}$ \\
\hline
$12n_{280}$ & $2t^{-1}-4+ 7t-8t^2+ 9t^3-9t^4+ 6t^5-4t^6+ 2t^7$\\
\hline
$12n_{323}$ & $-2t^{-5}+ 4t^{-4}-6t^{-3}+ 9t^{-2}-9t^{-1}+ 9-7t+ 5t^2-2t^3$\\
\hline
$12n_{356}$ & $2t^{-4}-5t^{-3}+ 8t^{-2}-10t^{-1}+ 11-10t+ 8t^2-5t^3+ 2t^4$\\
\hline
$12n_{375}$ & $2t^2-4t^3+ 8t^4-9t^5+ 10t^6-10t^7+ 7t^8-5t^9+ 2t^{10}$\\
\hline
$12n_{452}$ & $2t^{-1}-4+ 7t-9t^2+ 10t^3-9t^4+ 7t^5-5t^6+ 2t^7$\\
\hline
$12n_{706}$ & $2t^{-4}-4t^{-3}+ 6t^{-2}-8t^{-1}+ 9-8t+ 6t^2-4t^3+ 2t^4$ \\
\hline
$12n_{729}$ & $3t^2-6t^3+ 10t^4-12t^5+ 13t^6-12t^7+ 9t^8-6t^9+ 2t^{10}$\\
\hline
$12n_{811}$ & $-2+ 6t-8t^2+ 11t^3-11t^4+ 10t^5-8t^6+ 5t^7-2t^8$\\
\hline
$12n_{873}$ & $3t^{-4}-7t^{-3}+ 11t^{-2}-14t^{-1}+ 15-14t+ 11t^2-7t^3+ 3t^4$\\
\hline
\end{tabular}
\vspace{5pt}
\caption{Knots with twelve crossings appearing in the KnotInfo database \cite{KnotInfo} that Theorem \ref{theorem:AlmostAlternating} implies have $g_T(K)>1$ and $\dalt(K)>1$. Work of Jablan \cite{Jablan:Turaev} and Howie shows that for each of these knots $K$, we have $g_T(K)=\dalt(K)=2$. }
\label{table:12KnotsJones}
\end{table}

\clearpage
\bibliography{sigbib}{}
\bibliographystyle {amsalpha}
\end{document}